\documentclass[12pt]{amsart}

\usepackage{amsmath}
\usepackage{amsthm}
\usepackage{amssymb}
\usepackage{amscd} 
\usepackage{amsfonts} 

\newtheorem{theorem}{Theorem}[section] 

\newtheorem{prop}[theorem]{Proposition}
\newtheorem{proposition}[theorem]{Proposition}
\newtheorem{lemma}[theorem]{Lemma}

\newtheorem{corollary}[theorem] {Corollary} 

\newtheorem*{FTT}{Flat Torus Theorem}

\theoremstyle{definition}
\newtheorem{remark}[theorem]{Remark}

\def\modS{{\rm{Mod}}(S)}

\def\E{\mathbb E}
\def\G{\Gamma}
\def\Z{\mathbb Z}
\def\R{\mathbb R} 
\def\r{\rho}
\def\l{\lambda} 
\def\a{\alpha}
\def\b{\beta}
\def\c{\gamma}
\def\d{\delta}  
\def\g{\gamma}

\begin{document} 

\title[Actions of ${\rm{Aut}}(F_n)$
on CAT$(0)$ spaces]{The rhombic dodecahedron and semisimple actions of ${\rm{Aut}}(F_n)$
on CAT$(0)$ spaces}
\author{Martin R.~Bridson}
\address{Martin R.~Bridson\\
Mathematical Institute \\
24--29 St Giles'\\
Oxford OX1 3LB  \\ 
U.K. }
\email{bridson@maths.ox.ac.uk}

\date{22 Feb 2011}

\thanks{The work presented in this paper
was funded by a Senior Fellowship
from the EPSRC of Great Britain and a Wolfson Research Merit Award from the Royal
Society. A first draft was written while the author was on leave from Imperial College London, visiting the EPFL and the University of Geneva. The author thanks all of these organisations.}
  
\maketitle 
 
\begin{abstract} We consider actions of automorphism
groups of free groups by semisimple isometries on complete CAT$(0)$ spaces.
If $n\ge 4$ then each of the
Nielsen generators of {\rm{Aut}}$(F_n)$ has a fixed point. 
If $n=3$ then either each of the
Nielsen generators has a fixed point, or else they are hyperbolic and each
Nielsen-generated $\Z^4\subset{\rm{Aut}}(F_3)$ leaves invariant 
an isometrically embedded copy of Euclidean 3-space $\E^3\hookrightarrow X$ on
which it acts as a discrete group of translations with the
rhombic dodecahedron as a fundamental domain. An abundance
of actions of the second kind is described. 

Constraints on maps from {\rm{Aut}}$(F_n)$ to mapping class groups and linear
groups are obtained. If $n\ge 2$ then neither {\rm{Aut}}$(F_n)$
nor {\rm{Out}}$(F_n)$ is the fundamental group of a compact K\"ahler manifold.
\end{abstract}

{\em For Mike Davis on his 60th birthday, with affection and great respect}

\section{Introduction} 

This article is part of a project to understand the ways in which mapping class groups
and automorphism groups of free groups can act on {\rm{CAT$(0)$}} spaces.
Here we focus mainly (but not exclusively) on actions that are by semisimple isometries.
This includes all cellular actions on polyhedral complexes with only 
finitely many isometry types of cells \cite{mb:poly}.

The action of {\rm{Aut}}$(F_n)$ on the abelianization of $F_n$ gives an epimorphism 
${\rm{Aut}}(F_n)\to {\rm{GL}}(n,\Z)$. The inverse image of  ${\rm{SL}}(n,\Z)$
is generated by {\em{Nielsen transformations}}, which are the obvious lifts of
the elementary matrices:  fixing a
basis $\{a_1,\dots,a_n\}$ for $F_n$, one defines the left Nielsen transformations
 $\l_{ij}$ by $[a_i\mapsto a_ja_i,\, a_k\mapsto a_k (k\ne i)]$
and the right Nielsen transformations $\r_{ij}$ by $[a_i\mapsto a_ia_j,\, a_k\mapsto a_k (k\ne i)]$.

Semisimple isometries of  {\rm{CAT$(0)$}} spaces divide into  elliptics (those with fixed points) and
hyperbolics (those that have a non-trivial axis of translation); conjugate isometries are of 
the same type.
The Nielsen transformations are all conjugate in {\rm{Aut}}$(F_n)$, so the
semisimple actions of {\rm{Aut}}$(F_n)$ on {\rm{CAT$(0)$}} spaces divide into two classes: those where
the Nielsen transformations act as hyperbolic isometries and those where
they act as elliptic isometries. We shall prove that if $n\ge 4$
then there are no actions of the former type. The proof (which establishes
something more -- Proposition \ref{p:more}) is based on an idea of
Gersten \cite{gersten}.

\begin{theorem}\label{t:ell} If $n\ge 4$, then each Nielsen transformation in {\rm{Aut}}$(F_n)$ fixes a point
whenever {\rm{Aut}}$(F_n)$ acts by semisimple isometries on a complete 
{\rm{CAT$(0)$}} space.
\end{theorem}

If the dimension of the {\rm{CAT$(0)$}} space is sufficiently small, one can promote the
existence of fixed points for the individual Nielsen transformations to a
fixed point for the whole group -- see \cite{mb:helly}, \cite{mrb:crelle}. And by considering induced actions, 
one can extend Theorem \ref{t:ell} to finite-index subgroups  $\G<{\rm{Aut}}(F_n)$: any
power of a Nielsen transformation that lies in $\G$ fixes a point whenever $\G$ acts by semisimple
isometries on a complete  {\rm{CAT$(0)$}} space (see Section \ref{s:induce}). 
 In particular, if $\lambda_{ij}^p\in \G$ then $\lambda_{ij}^p$ must have finite order in $H_1(\G,\Z)$;
for if not then there would be a homomorphism $\G\to\Z$ mapping  $\l_{ij}^p$ 
non-trivially.

Despite recent progress in the understanding of maps between automorphism
groups of free groups and mapping class groups, 
many issues remain unresolved. For example, it is unknown
whether every homomorphism from {\rm{Aut}}$(F_n)$ to a mapping class group
has to have finite image if $n\ge 4$. Closed surfaces
of negative Euler characteristic admit semisimple actions on complete {\rm{CAT$(0)$}} spaces
where the elliptic isometries are the roots of multi-twists (\cite{mb:bill}, Theorem A), so
Theorem \ref{t:ell} constrains putative maps ${\rm{Aut}}(F_n)\to\modS$.

\begin{corollary}\label{c:mcg} If $n\ge 4$ and $\Gamma< {\rm{Aut}}(F_n)$
is a subgroup of finite index, then every homomorphism from $\Gamma$ to a mapping
class group sends powers of Nielsen transformations to roots of  (possibly trivial)
multi-twists.
\end{corollary}

We shall prove in Theorem \ref{t:3toMCG} that this corollary fails when $n=3$.

\smallskip

The situation for $n=3$ is completely different. In the course of
their work on linear representations of ${\rm{Aut}}(F_n)$, 
Grunewald and Lubotzky \cite{GL} proved that ${\rm{Out}}(F_3)$ has a subgroup
of finite index that maps onto a non-abelian free group. 
This gives rise
to an abundance of semisimple actions of {\rm{Aut}}$(F_3)$ 
and {\rm{Out}}$(F_3)$ in which the Nielsen transformations
act hyperbolically, as we shall explain in Section \ref{s:N-hyp}. 
But across this enormous range of 
actions there is a beautiful geometric invariant associated to the
abelian subgroups of maximal rank generated by Nielsen transformations. Each
such subgroup is conjugate\footnote{${\rm{Aut}}(F_3)$ has virtual cohomological dimension 4,
so there are no free abelian subgroup of greater rank, but there is one other type of $\Z^4$,
as explained on page 154 of \cite{FH}.}
 to $\Lambda:=\langle \l_{21},\, \r_{21},\,\l_{31},\,\r_{31}\rangle \cong\Z^4$; we call 
 it a {\em{Nielsen  $\Z^4$}}.
 
 The {\em{rhombic dodecahedron}} is the Catalan solid that is
dual to the Archimedean cuboctahedron; it is described in more detail in
Section \ref{s:ddh}. Notice that in the following theorem no assumption is
made about the discreteness of the action.

\begin{theorem}\label{t:ddh}
Whenever ${\rm{Aut}}(F_3)$ acts
by semisimple isometries on a complete {\rm{CAT$(0)$}} space $X$, either
each Nielsen transformation fixes a point, or else each Nielsen $\Z^4\subset  {\rm{Aut}}(F_3)$
leaves invariant an isometrically embedded 3-flat $\E^3\hookrightarrow X$ on which
it acts as a discrete group of translations with Dirichlet domain a rhombic dodecahedron. 
\end{theorem}

When $n\ge 6$ one can sharpen Theorem \ref{t:ell} by proving that in any action of
{\rm{Aut}}$(F_n)$ by isometries on a complete {\rm{CAT$(0)$}} space the Nielsen transformations have zero
translation length, i.e. they are either elliptic or else they are neutral parabolics. 
Actions of the second type arise from linear representations ${\rm{Aut}}(F_n)\to {\rm{GL}}(d,\R)$
via the action of ${\rm{GL}}(d,\R)$ on its symmetric space. The fact that Nielsen
transformations must act as neutral parabolics imposes constraints on the representation
theory of {\rm{Aut}}$(F_n)$ that supplement  
\cite{DF} and \cite{rapin}.
In Section \ref{s:kahler} we explain how the rigidity of the standard linear
representation \cite{DF} can be combined with 
Simpson's results \cite{carlos} to prove:

\begin{theorem}\label{t:kahler} If $n\ge 2$, then neither {\rm{Aut}}$(F_n)$ nor {\rm{Out}}$(F_n)$ is the
fundamental group of a compact K\"ahler manifold.
\end{theorem}

 I thank Michael Handel, Mustafa Korkmaz,
Nicolas Monod and Carlos Simpson  for helpful conversations related to this article.
I thank Dawid Kielak and Ric Wade for reading the manuscript carefully, and
 I thank Tadeusz Januszkiewicz for his patient editing.

\section{Isometries of CAT$(0)$ spaces and centralizers}

Our basic reference for {\rm{CAT$(0)$}} spaces and their isometries is \cite{BH}. Let $X$
be a complete {\rm{CAT$(0)$}} space.
The {\em{translation length}} of an isometry $\g\in{\rm{Isom}}(X)$
is $\|\g\|:=\inf\{d(x,\g.x) \mid x\in X\}$, and ${\rm{Min}}(\g):=\{x\in X\mid d(x,\g.x)=\|\g\|\}$.
An isometry is termed {\em{semisimple}} if ${\rm{Min}}(\g)$ is non-empty.
Semisimple isometries divide into {\em{elliptics}} (when $\|\g\|=0$) and {\em{hyperbolics}} 
($\|\g\|>0$). Isometries for which ${\rm{Min}}(\g)=\emptyset$ are termed {\em{parabolic}} ({\em{neutral}}
or {\em{non-neutral}} according to whether $\|\g\|=0$).

The following basic result is proved on page 231 of \cite{BH}.

\begin{proposition}\label{p:split}
If $\g$ is hyperbolic, then ${\rm{Min}}(\g)$ splits isometrically as $Y\times\R$, where $\g$
acts trivially on the first factor and as $(t\mapsto t+\|\g\|)$ on the second factor. 
If $\alpha\in{\rm{Isom}}(X)$ commutes with $\g$, then it leaves ${\rm{Min}}(\g)$
invariant, preserves its splitting, and acts by translation on the second factor.
\end{proposition}

We also need the following form of the Flat Torus Theorem (\cite{BH}, page 254).
A {\em $k$-flat} is an isometrically embedded copy of $k$-dimensional Euclidean space.

\begin{FTT} Whenever $\Z^r$ acts by semisimple isometries
on a complete {\rm{CAT$(0)$}} space $X$, it leaves invariant a $k$-flat  
$E\hookrightarrow X$ (some $k\le r$) on which it acts by translations.
If $\g\in\Z^r$ acts hyperbolically on $X$, then $\g|_E$ is a translation
of length $\|\g\|$.
\end{FTT}

\noindent{\em{Notation:}} Let $\g$ and $\d$ be commuting hyperbolic isometries. 
We write $\g\perp\d$ if at some (hence any) point $x\in{\rm{Min}}(\g)\cap{\rm{Min}}(\d)$
the axes of $\g$ and $\d$ through $x$ are orthogonal. Equivalently, the
action of $\d$ on the second factor of ${\rm{Min}}(\g) = Y\times\R$ is trivial.

We write $Z_G(S)$ for the centralizer of a set $S$ in a group $G$. We
write $H_1(G)$ for the abelianisation of $G$ and $[G,G]$ for the
commutator subgroup.

\begin{lemma}\label{l:null}
Suppose that $\G$ acts by isometries on a complete
CAT$(0)$ space $X$ and that $\gamma,\d\in\G$ are commuting  hyperbolic
isometries.
\begin{enumerate}
\item The natural map $\langle\gamma\rangle \to
H_1(Z_\G(\gamma))$ is injective.
\item If $\d\in [Z_\G(\g),Z_\G(\g)]$, then $\g\perp\d$.
\item If there exists $g\in Z_\G(\g)$ such that $g^{-1}\d^{-1} g = \d$
then $\g\perp \d$.
\end{enumerate}
\end{lemma}

\begin{proof} The action of $Z(\gamma)$ on the second factor of 
${\rm{Min}}(\gamma) = Y\times\R\subset X$ is by translations. The
group of all such translations is torsion-free and abelian, and the image of
$\gamma$ is non-trivial. This proves (i) and (ii), and (iii) is a special case of
(ii) since $\d^2 = g^{-1}\d^{-1} g \d\in [Z_\G(\g),Z_\G(\g)]$ and any axis for
$\d$ is an axis for $\d^2$.
\end{proof} 

\subsection{Induced actions} \label{s:induce}

Let $G$ be a group and let $H\subset G$ be a subgroup of index $d$.
Just as one induces an $n$-dimensional linear representation of $H$ to obtain
an $nd$-dimensional representation of $G$, so one can induce an action
of $H$ by isometries of a metric space $X$ to obtain an action of $G$ by isometries
on $X^d$ (with the product metric). One way to view this induction (following
\cite{tomD} p.35) is to identify $X^d$ with the space of $H$-equivariant maps
$f:G\to X$ with the action $(g.f)(\gamma) = f(\gamma g^{-1})$. 

The following lemma is covered in \cite{BH}, p.231--232.

\begin{lemma}\label{l:induced}
 Suppose $H$ has index $d$ in $G$ and that $H$ acts by isometries
on a complete {\rm{CAT$(0)$}} space $X$. If $g^p\in H$ acts as a hyperbolic isometry, where
$p$ is a non-zero integer, then $g$ acts hyperbolically in the induced action of $G$
on $X^d$.
\end{lemma}

One can also express induction in 
its group-theoretic form, the {\em{wreath product}}.
Recall that $A\wr B$ is the semidirect product $W=B\ltimes\oplus_{b\in B} A_b$, where the
$A_b$ are isomorphic copies of $A$ permuted by left translation. $\oplus_{b\in B} A_b$
is called the base of the wreath product.
If $B$ is finite
and $A$ acts on $X$ by isometries, then there is an obvious action of $W$ by isometries
on $X^B$, with $B$ permuting the factors and $A_b$ acting as $A$ in the $b$-coordinate and
trivially in the others. The induction described above is then just a manifestation of the following
standard lemma. To avoid complications we assume that $H$ is normal in $G$.

\begin{lemma}\label{l:wr}
 If $H$ is normal in $G$ and $\phi : H\to Q$ is a homomorphism,
then there is a homomorphism $\Phi: G\to Q\wr (G/H)$ so that $\Phi(H)
\subset \oplus_{x\in G/H} Q_x$ and 
for all $h\in H$ the coordinate of $\Phi(h)$
in $Q_1$ is $\phi(h)$.
\end{lemma}

\section{Automorphisms that cannot act hyperbolically}

{\em{Conventions:}} Throughout this article, $F_n$ denotes a free group of 
rank $n$.
We work with the left action of ${\rm{Aut}}(F_n)$ on $F_n$ and the commutator convention
$[\alpha,\beta]=\alpha\beta\alpha^{-1}\beta^{-1}$. (So $[\a,\b]$ is the automorphism ``$\beta^{-1}$, followed by
$\alpha^{-1},\dots$")
We write ${\rm{ad}}_\gamma$ to denote the inner automorphism $x\mapsto 
\gamma x\gamma^{-1}$ (so ${\rm{ad}}_{\gamma\delta}={\rm{ad}}_{\gamma}\circ{\rm{ad}}_{\delta}$). Note that 
$\phi\,{\rm{ad}}_\gamma\,\phi^{-1} = {\rm{ad}}_{\phi(\gamma)}$.

\smallskip

The following lemma is a variation on the argument that Gersten \cite{gersten}
used to show that ${\rm{Aut}}(F_n)$ cannot act properly and cocompactly on a 
{\rm{CAT}}$(0)$  space if $n\ge 3$ (cf.~\cite{BH} p.253).

\begin{lemma}\label{l:pq}
Let $G_{p,q}=\langle \a,\b,\c,t\mid [t,\a],\, t\b t^{-1}=\b\a^p,\, t\c t^{-1}=\c\a^q\rangle $,
where $p$ and $q$ are non-zero integers. If $p\neq q$, then 
$\a$ fixes a point whenever $G_{p,q}$ acts by semisimple isometries on a 
complete {\rm{CAT}}$(0)$  space.
\end{lemma}

\begin{proof} Suppose $G_{p,q}$ acts 
semisimply on a complete {\rm{CAT}}$(0)$  space
$X$. 
Note that $\b$ conjugates $t$ to $\a^pt$ and $\c$ conjugates $t$ to
$\a^qt$, so $\|t\|=\|\a^pt\|=\|\a^qt\|$.
 Thus if $t$ is elliptic then $\a^pt$ is elliptic, hence
$\a^p=(\a^pt)t^{-1}$ as a product of commuting elliptics is elliptic,
and therefore $\a$ is elliptic. 

If $t$ is hyperbolic then  by the Flat Torus Theorem, $A=\langle \a,t\rangle $ would act by translations
on a geodesic line or Euclidean plane $E$ in ${\rm{Min}}(t)\cap{\rm{Min}}(\a)$ with $t,\, t\a^p,\, t\a^q$ all acting as translations
of length $\|t\|$. 
But the images of any point $e\in E$
under these three translations are colinear. This is incompatible with the
fact that they are equidistant from $e$ unless we are in the degenerate situation where the points coincide, i.e. $\a$ is acting trivially on $E$ and hence is elliptic.
\end{proof}  

For the next lemma it is convenient to work with a basis $\{a_0,\dots,a_n\}$ for $F_{n+1}$,
setting $F_n=\langle a_1,\dots,a_n\rangle $, and for each $w\in F_n$
defining $R_w\in {\rm{Aut}}(F_{n+1})$ by $R_w(a_0)=a_0w$ and
$R_w(a_i)=a_i$ for $i>0$. 

$R_{a_i}$ is a Nielsen transformation, so the
following proposition implies Theorem \ref{t:ell}.

\begin{proposition} \label{p:more}
If $w\in F_n$ lies 
in a free factor of rank $n-2$, then $R_w$ fixes a point whenever
${\rm{Aut}}(F_{n+1})$ acts by semisimple isometries on a complete {\rm{CAT}}$(0)$  space.
\end{proposition}

\begin{proof} Without loss of generality we may assume that $w
\in\langle a_1,\dots,a_{n-2}\rangle $. Let $p\neq q$ be non-zero
integers and let $T\in{\rm{Aut}}(F_{n+1})$ be the automorphism
defined by $T(a_{n-1})=a_{n-1}w^p,\,T(a_n)=a_nw^q$ and
$T(a_i)=a_i$ for $i=0,\dots,n-2$. An elementary calculation shows
that the assignment $[\a\mapsto R_w,\, \b\mapsto R_{a_{n-1}},\,
\c\mapsto R_{a_n},\, t\mapsto T]$ respects the defining 
relations of the group $G_{p,q}$ of Lemma \ref{l:pq}
and hence defines a homomorphism
$G_{p,q}\to {\rm{Aut}}(F_{n+1})$. Thus any action of ${\rm{Aut}}(F_{n+1})$ by
semisimple isometries on a complete {\rm{CAT}}$(0)$  space gives rise to an
action of $G_{p,q}$, and Lemma \ref{l:pq} tells us that
$R_w$, the image of $\alpha\in G_{p,q}$ must act elliptically. 
\end{proof}

Our interest in the following lemma lies with the case $n=3$.

\begin{lemma}\label{l:inner} Let $n\ge 3$.
 Whenever ${\rm{Aut}}(F_n)$ acts by semisimple isometries
on a complete {\rm{CAT}}$(0)$  space, the inner automorphisms corresponding to
primitive elements are elliptic.
\end{lemma}

\begin{proof} If $n\ge 4$, then the preceding proposition implies that the right Nielsen transformations
$\r_{ij}$ are elliptic, hence so are their conjugates $\l_{ij}$. The inner automorphism
corresponding to a primitive element,
say $a_1$ from the basis $\{a_1,\dots,a_n\}$, is a composition of commuting Nielsen transformations:
${\rm{ad}}_{a_1} = (\l_{21}^{-1}\r_{21})\dots (\l_{n1}^{-1}\r_{n1})$, and a product of commuting elliptic
isometries is elliptic.

When $n=3$ the Nielsen transformations need not be elliptic, so we require a different argument.
Let $\{a,b,c\}$ be a basis for $F_3$
and let $\tau\in{\rm{Aut}}(F_3)$ be the automorphism
$[a\mapsto a,\, b\mapsto ba^p,\, c\mapsto ca^q]$.
The assignment $[\a\mapsto {\rm{ad}}_a,\, \b\mapsto{\rm{ad}}_b,\, 
\c\mapsto{\rm{ad}}_c,\, t\mapsto \tau]$ defines a homomorphism
$G_{p,q}\to{\rm{Aut}}(F_3)$ sending $\a$ to ${\rm{ad}}_a$, so Lemma
\ref{l:pq} implies that ${\rm{ad}}_a$ must be elliptic.
\end{proof}

\begin{remark} The maps $G_{p,q}\to{\rm{Aut}}(F_{n+1})$ 
and $G_{p,q}\to{\rm{Aut}}(F_{3})$ used in the preceding proofs are injective, but since we do not need this fact we omit the proof.
\end{remark}

\section{The rhombic dodecahedron}\label{s:ddh}
The {\em{rhombic dodecahedron}} is one of the Catalan solids; it is
dual to the Archimedean solid known as the cuboctahedron. It has
fourteen vertices and twelve faces. The faces are rhombi 
whose diagonals have lengths  in the ratio $1:\sqrt 2$. 

Consider the standard tiling of $\R^3$ by unit cubes and fix a 
vertex $u$. The rhombic dodecahedron can be constructed as the convex hull
of the endpoints of the six edges incident at $u$ together
with the centres of the eight cubes incident at $u$. In other words,
 it is the Voronoi cell for the face-centred cubic lattice (and therefore occurs
 naturally in many crystal formations).
 
 It follows from this description that the rhombic dodecahedron 
is a Dirichlet domain for the free abelian group of rank 3
that translates $\R^3$ by integer vectors $(a,b,c)$ with
$a+b+c$ even.  This is the group generated by the edge-vectors 
of an octahedron dual to a cube with edge-length 2.
For our purposes, it is convenient to rephrase this as follows.

\begin{lemma}\label{l:octo} Let $A\subset\R^3$ be the $\Z$-module
generated by unit vectors $u_1,u_2,v_1,v_2$. If
\begin{enumerate}
\item $u_1+u_2=v_1+v_2$,
\item $u_1\perp u_2$ and $v_1\perp v_2$, and 
\item $(u_1-u_2)\perp (v_1-v_2)$,
\end{enumerate}
then $A$ is discrete and
 $\{ x\mid \|x\| \le \|x-a\|\text{ for all }a\in A\}$ is a rhombic dodecahedron.
\end{lemma}

\begin{remark}\label{rem}
Note that the conclusion of the lemma includes the observation that $A$ is not
contained in a plane. Thus one cannot find  such vectors 
$u_1,u_2,v_1,v_2$ in $\R^d$ for $d<3$.
\end{remark}

\begin{proposition}\label{p:makeDDH} Let $\G$ be a group. Suppose that
$\a_1,\a_2,\b_1,\b_2$ commute and are conjugate in $\G$. Let
$A=\langle \a_1,\a_2,\b_1,\b_2\rangle $ and suppose that 
\begin{enumerate}
\item[(i)] $\a_1\a_2=\b_1\b_2$, 
\item[(ii)] $\a_1\in [Z_\G(\a_2), Z_\G(\a_2)]$ and $\b_1\in [Z_\G(\b_2), Z_\G(\b_2)]$, and 
\item[(iii)] there exists $\g\in Z_\G(\a_1\a_2^{-1})$ with $\g^{-1}(\b_1\b_2^{-1})\g=\a_2\a_1^{-1}$.
\end{enumerate}
Then, whenever $\G$ acts by semisimple isometries on a complete {\rm{CAT}}$(0)$  space $X$, either
$A$ fixes a point, or $A$ leaves invariant an isometrically embedded 3-flat $\E^3\hookrightarrow X$, acting properly and cocompactly on it by isometries, with Dirichlet
domain a rhombic dodecahedron.
\end{proposition}

\begin{proof} Since the given generators of $A$ are all conjugate, either all are
elliptic or all are hyperbolic. If they are elliptic, then they have a common fixed point, since
they commute. If they are hyperbolic, then by the Flat Torus Theorem
 $A$ leaves
invariant a $k$-flat $E\subset X$, with $k\le 3$, and acts on it by translations;
since the $a_i$ and $b_i$ are hyperbolic, they act non-trivially on $E$, and since
they are conjugate in $\G$, their translations lengths are equal. 
Let $u_i$ (resp.~$v_i$) be the translation vector of $\a_i|_E$ (resp.~$\b_i|_E$). 
According to
Lemma \ref{l:null}, 
condition (ii) implies that $u_1\perp u_2$ and $v_1\perp v_2$ and condition (iii) implies that $(u_1-u_2)\perp (v_1-v_2)$. 
Lemma \ref{l:octo} (with the remark that follows it) completes the proof.
\end{proof}

\begin{remark} The above proposition admits obvious variations in which
conditions (ii) and (iii) are altered to allow parts (2) and (3)
of Lemma \ref{l:null} to be applied
in different combinations.
\end{remark}

\section{The shape of Nielsen-hyperbolic actions of {\rm{Aut}}$(F_3)$}

The free abelian subgroups of Aut$(F_3)$ have rank at most 4 (the virtual cohomological dimension of Aut$(F_3)$). Each Nielsen $\Z^4$ is conjugate
to $\Lambda = \langle \l_{21},\,\r_{21},\,\l_{31},\,\r_{31}\rangle $.

Proposition \ref{p:makeDDH} applies directly to the image of $\Lambda$ in ${\rm{Out}}(F_3)$.

\begin{lemma}\label{l:out3} In $\G={\rm{Out}}(F_3)$, the elements $\a_1:=\l_{21}^{-1},\,
\a_2:=\r_{21},\, \b_1=\l_{31}^{-1},\, \b_2=\r_{31}$ and $\g = \varepsilon_2$
satisfy the conditions of Proposition \ref{p:makeDDH}.
\end{lemma}

\begin{proof} In ${\rm{Out}}(F_3)$ we have ${\rm{ad}}_{a_1} = \l_{21}^{-1}\r_{21}\l_{31}^{-1}\r_{31}=1$,
so $\a_1\a_2=\b_1\b_2$.

A direct calculation yields the well-known relation
$[\l_{23}^{-1},\, \l_{31}^{-1}] = \l_{21}^{-1}$, 
and both $\l_{23},\, \l_{31}$ commute with
$\r_{21}$. Thus $\l_{21}$ lies in the commutator subgroup of $\r_{21}$. Likewise,
$[\r_{23}^{-1},\, \r_{31}^{-1}] = \r_{21}^{-1}$ implies that 
$\r_{21}$ is in the commutator subgroup of $\l_{21}$.

Conjugation by $\varepsilon_2$ leaves $\l_{31}$ and $\r_{31}$ fixed while
interchanging $\l_{21}^{-1}$ and $\r_{21}$.
\end{proof}

Proposition \ref{p:makeDDH} does not 
apply directly to $\Lambda\subset{\rm{Aut}}(F_3)$ but the difficulty is a minor one.
 
 \medskip
 
\noindent{\em{Proof of Theorem \ref{t:ddh}}}.
Suppose that ${\rm{Aut}}(F_3)$ is acting by isometries on a complete {\rm{CAT}}$(0)$  space $X$ with the Nielsen
moves $\l_{ij}$ acting as hyperbolic isometries. According to 
the Flat Torus Theorem, 
$\Lambda\cong\Z^4$ leaves invariant a $k$-flat $E\subset X$
on which it acts by translations, with the generators $\l_{ij}$ and $\r_{ij}$ 
(which are conjugate in ${\rm{Out}}(F_3)$) acting non-trivially by translations of the same length. 
Lemma \ref{l:inner} tells us that ${\rm{ad}}_{a_1}= \l_{21}^{-1}\r_{21}\l_{31}^{-1}\r_{31}$ acts elliptically 
on $X$ and hence trivially on $E$. Replacing $E$ by the convex hull of a $\Lambda$-orbit,
we may assume that it has dimension at most 3.

The calculations in the proof of Lemma \ref{l:out3} allow us to apply 
Lemma \ref{l:null}: as in the proof of  Proposition \ref{p:makeDDH}, it
implies that the translation vectors of $\l_{21}|_E,\, \r_{21}|_E,\, \l_{31}|_E,\, \r_{31}|_E$
satisfy the conditions of Lemma \ref{l:octo}. Hence $E$ has dimension 3
(remark \ref{rem}), the
action of $\Lambda/\langle {\rm{ad}}_{a_1}\rangle $ is proper and cocompact, and a Dirichlet domain
for the action is a rhombic dodecahedron. \qed

\section{An abundance of actions when $n=3$}\label{s:N-hyp}

Let $\G$ be an arbitrary finitely generated group. In this
section we shall assign to each action of $\G$ on a {\rm{CAT$(0)$}
space $X$ actions of Aut$(F_3)$ and ${\rm{Out}}(F_3)$ on the Cartesian product of finitely
many copies of $X$. The assignment is such that if the generators
of $\G$ act as hyperbolic
isometries, then the induced actions of Aut$(F_3)$ and ${\rm{Out}}(F_3)$ will be Nielsen-hyperbolic.

The heart of the construction is the following proposition.

\begin{prop}\label{p:aut3toF}
For each positive integer $k$, there exists a subgroup of 
finite index $H_k\subset{\rm{Out}}(F_3)$ and a surjective 
homomorphism $\pi_k:H_k\to F_k$
such that the image of $\pi_k$ is generated by the images of powers of 
conjugates of Nielsen transformations $\l_{ij}$. 
\end{prop}

\begin{proof} Fix a basis $\{a_1,a_2,a_3\}$ for $F_3$.
We first construct a map from a subgroup of finite index
in ${\rm{Out}}(F_3)$ to a free group of rank 2 so that the image is
generated by powers of $\l_{12}$ and $\l_{21}$.
This map is based on a construction of Grunewald and Lubotzky \cite{GL} (cf.~\cite{BVsurv}, Qu.9).

Regard $F_3$ as the fundamental group of a graph $R$ with
one vertex. The loops of length one are
labelled $\{a_1,a_2,a_3\}$. Consider the 2-sheeted
covering $\hat R\to R$ with fundamental group
$\langle a_1,a_2,a_3^2,a_3a_1a_3^{-1},a_3a_2a_3^{-1}
\rangle$ and let $G\subset{\rm{Aut}}(F_3)$ be the stabilizer of
this subgroup. 

$G$ acts on $H_1(\hat R,\mathbb Q)$
leaving invariant
the eigenspaces of the involution that generates the Galois group of the
covering. The eigenspace corresponding to the eigenvalue $-1$ is two
dimensional with basis $\{a_1-a_3a_1a_3^{-1},\, 
a_2-a_3a_2a_3^{-1}\}$. The action of $G$
with respect to this basis gives  an epimorphism $G\to{\rm{GL}}(2,\Z)$.
By replacing $G$ with a subgroup of finite index one
can ensure that the inner automorphisms in $G$ act trivially on 
$H_1(\hat R,\mathbb Q)$ (because ${\rm{Inn}}(F_n)\cap G$ has finite
image in ${\rm{GL}}(2,\Z)$, which is virtually torsion-free). 
 If $Q$ is the image of $G$
in ${\rm{Out}}(F_3)$, then we have a
map $\mu: Q\to{\rm{GL}}(2,\Z)$ whose image is of finite index. 

Since ${\rm{GL}}(2,\Z)$ has a free subgroup of
finite-index, by passing to a further subgroup if necessary we may assume
that $\mu(Q)$ is free. 
We choose $p$ so that  $\l_{12}^p,\l_{21}^p\in Q$. The images of  
 $\mu(\l_{12})$ and  $\mu(\l_{21})$ generate a subgroup of finite index in
${\rm{GL}}(2,\Z)$, so the subgroup $L_2$ generated
by $a:=\mu(\l_{12}^p)$ and $b:=\mu(\l_{21}^p)$ is free of rank 2. Applying Marshall Hall's
theorem, we pass to a subgroup of finite index in $\mu(Q)$ that retracts
onto  $L_2$. Let $H_2$ be the inverse image of this subgroup in
$Q$ and let $\pi:H_2\to L_2$ be its surjection to $L_2$. 

In $L_2=\langle a,b\rangle $ one has the subgroup $L_k$ of index $k-1$
with basis $\{a^iba^{-i},\,a^{k-1}\mid i=0,\dots,k-2\}$; this
is free of rank $k$ and is generated by conjugates of powers
of $a$ and $b$. Thus it suffices to take $H_k = \pi^{-1}(L_k)$. 
\end{proof}
 
\subsection{Induced actions}
Let $\G$ be a group generated by the image of an epimorphism
$\psi:F_r\to\G$ from a finitely generated free group. Suppose that $\G$
acts by isometries on a complete {\rm{CAT$(0)$} space $X$.
By composing 
$\psi$ with the map $H_r\to F_r$ constructed
in the preceding proposition, we obtain a surjection $H_r\to\G$ sending
powers of certain Nielsen transformations to generators of $\G$. And
by inducing the action we obtain an action of
{\rm{Aut}}$(F_3)$ on a product of finitely many copies of $X$. If 
the action of $\G$ on $X$ was by hyperbolic isometries, then the 
Nielsen transformations act as hyperbolic isometries in this induced
action. 

\section{Homomorphisms from ${\rm{Out}}(F_3)$ to mapping class groups}
\label{s:mcg3}

We saw in the introduction that when $n\ge 4$,
any homomorphism from ${\rm{Aut}}(F_n)$
to a mapping class group must send Nielsen transformations to roots of multi-twists.
We also noted that no homomorphisms with infinite image are known to exist.
The situation for $n=3$ is completely different.
Let $S_g$ be the closed surface of genus $g$ and let ${\rm{Mod}}(S_g)$
be its mapping class group. 
Let
${\rm{Mod}}(S_{g,1})$ be the mapping class group
of the genus $g$ surface with one boundary component.

\begin{theorem} \label{t:3toMCG}
For certain positive integers $g$, there
exist homomorphisms ${\rm{Out}}(F_3)\to{\rm{Mod}}(S_g)$ 
sending Nielsen transformations to elements of infinite order that are not roots of
multi-twists.
\end{theorem}
 
I do not know the value of the least integer
$g$ for which there is such a homomorphism (and likewise for  ${\rm{Aut}}(F_3)$).

\medskip

\begin{proposition}\label{l:sg} For all positive integers $h$ and all
non-trivial finite groups $G$, there exist integers $g$ for which there is an injective homomorphism 
${\rm{Mod}}(S_{h,1})\wr G \to {\rm{Mod}}(S_{g})$.
\end{proposition}
 
\begin{proof} Every finite group has a faithful realisation 
as a group of symmetries of a closed hyperbolic surface (and therefore embeds in the mapping class group of that surface). 
We realise $G$ on the surface $Y$ and equivariantly delete an open
disc about each point in a free orbit. We then glue a copy of $S_{h,1}$
to each of the  resulting boundary components  and extend
the action of $G$ in the obvious manner. Let $S_g$ be the resulting surface.

Corresponding to each of the attached copies of $S_{h,1}$ there is an
injective homomorphism ${\rm{Mod}}(S_{h,1}) \to {\rm{Mod}}(S_{g})$ 
obtained by extending
homeomorphisms to be the identity on the complement of the attached copy
of $S_{h,1}$. (The extension respects isotopy classes.) Thus we obtain
$|G|$ commuting copies $\{M_\g\mid \g\in G\}$
 of ${\rm{Mod}}(S_{h,1})$ in ${\rm{Mod}}(S_{g})$. The
action of $G\subset {\rm{Mod}}(S_{g})$ by conjugation permutes the $M_g$
and the canonical map $G\ltimes \otimes_{\g\in G} M_\g\to {\rm{Mod}}(S_{g})$
 is injective.
\end{proof}
 
\noindent{\em{Proof of Theorem \ref{t:3toMCG}:}} Let $F_r$ be a free group of
rank $r$  that maps onto $\G={\rm{Mod}}(S_{h,1})$ sending at least one basis
element, say $a_1$,
to an element  of infinite order $\psi\in\G$ that is not a root of a multi-twist.  
Proposition \ref{p:aut3toF}  
provides a normal subgroup of finite index $H_r\subset{\rm{Out}}(F_3)$
 that maps onto $F_r$
sending a power of a Nielsen transformation, say $\l^p$,
 to $a_1$. Thus we obtain a homomorphism $H_r\to \G$ sending $\l^p$ to $\psi$.
 We replace $H_r$ with a subgroup of finite index  that is normal in
${\rm{Out}}(F_3)$ (with quotient $\Omega$, say)
and consider the induced homomorphism
${\rm{Out}}(F_3)\to\G \wr \Omega$ as in Lemma \ref{l:wr}.   As in Proposition
 \ref{l:sg},
this wreath product acts on a closed surface $S_g$. By construction,
a non-zero power of
$\l^p$ acts on one of the subsurfaces $S_{h,1}$ as a non-zero power of $\psi$
and hence the image of  $\l$ in  ${\rm{Mod}}(S_{g})$  is not a root of a multi-twist.
\qed
 
\section{Nielsen generators are not ballistic if $n\ge 6$}

\begin{lemma}\label{l:dehn} Let
$\lambda\in{\rm{Aut}}(F_n)$ be a  Nielsen generator and let $Z$ be its centralizer.
If $n\ge 6$, then $[\lambda]=0$ in $H_1(Z,\Z)$. 
\end{lemma}

\begin{proof} For any $n\ge 2$ one can realise $\l$ as 
a Dehn twist $\delta$ about a non-separating curve on a compact orientable surface 
$S$ that has genus $\lfloor n/2\rfloor$ and euler characterisic $(1-n)$.  
More precisely, taking a basepoint on the boundary of $S$, 
there is an isomorphism\footnote{A 1-holed torus $T$ is the regular neighbourhood of 
the union $Y$ of 2 simple loops that intersect once; orient them, label them $a_1,a_2$
and join their intersection point to the basepoint by an arc that does not intersect the loops;
 the automorphism of $\pi_1T\cong F_2$ induced by the Dehn twist in $a_1$ is easily seen to
be one of $\l_{21}^{\pm 1}$ or $\r_{21}$ (the four possibilities corresponding to the choices 
of orientation  for the loops). For $n>2$, one attachs a suitable surface  to $T$ along an arc in its boundary that
includes the basepoint and extends $\{a_1,a_2\}$ to a basis for $\pi_1$ where the remaining
basis loops lie in the attached surface.}
from $\pi_1S$ to $F_n$ that induces a homomorphism  from the mapping class group $\modS$  to ${\rm{Aut}}(F_n)$ sending $\delta$ to $\lambda$.   

Let $C(\delta)$ be the centralizer of $\delta$ in  $\modS$. The map from $\langle \delta\rangle =\langle \lambda\rangle $ to $H_1(Z,\Z)$ factors
through $H_1(C(\delta),\Z)$, so it is enough to prove that $\delta$ has trivial image in $H_1(C(\delta),\Z)$. If $\delta$
is the twist in a loop $c$, then $C(\delta)$ is the image of the natural map to $\modS$ of the mapping class group of the
surface $S'$ obtained by cutting $S$ open along $c$. Thus it is enough to prove that in a surface $S'$ of genus at least $2$, the
Dehn twist in any loop $c$ parallel to a boundary curve  is trivial in ${\rm{Mod}}(S')$.
This can be seen using the lantern relation \cite{korkmaz}.
In more detail, the lantern
relation involves 7 loops on a 4-holed sphere; if one embeds the 4-holed sphere in $S'$ so that one of the  loops is
sent to $c$
and the remaining 6 loops are non-separating,  then the relation takes the
form  $\delta_0\delta_1\delta_2\delta_3 = \delta_4\delta_5\delta_6$,
where $\delta_0$ is the twist in $c$ and the remaining $\delta_i$ are positive twists in non-separating loops. 
The twists in any two non-separating loops are conjugate in ${\rm{Mod}}(S')$, hence equal (to $\tau$ say) in
$H_1({\rm{Mod}}(S'),\Z)$. So in $H_1$ the above relation becomes $\delta_0\tau^3=\tau^3$.
\end{proof}
 
\begin{proposition} Suppose $n\ge 6$. Whenever ${\rm{Aut}}(F_n)$ acts by isometries on a
complete {\rm{CAT}}$(0)$ space, the Nielsen generators $\l_{ij}$ and $\r_{ij}$ have zero
translation length (i.e.~they are either elliptic or else they are neutral parabolics).
\end{proposition}

\begin{proof} A lemma of
Karlsson and Margulis \cite{KM} can be used to show that
if an isometry $\gamma$ of a complete {\rm{CAT$(0)$}} space has positive
translation length then $\gamma$ must have infinite
order in the abelianisation of its centralizer (cf.~the proof of
\cite{mb:bill} Theorem 1). Lemma \ref{l:dehn} completes the proof.
\end{proof}

This result places constraints on the representation theory of ${\rm{Aut}}(F_n)$.
I shall return to this point in another article but note one example here for
illustrative purposes.

\begin{corollary}\label{c:evalues}
Let $\lambda\in{\rm{Aut}}(F_n)$ be a Nielsen transformation
and let $\Phi: {\rm{Aut}}(F_n)\to {\rm{SL}}(d,\R)$ be an arbitrary
representation. If $n\ge 6$, then the eigenvalues of $\Phi(\l)$ all have modulus $1$.
\end{corollary}

\begin{proof} The symmetric space ${\rm{SL}}(d,\R)/{\rm{SO}}(d,\R)$  is non-positively
curved and the action of $M\in {\rm{SL}}(n,\R)$ has positive translation
length if $M$ has an eigenvalue of modulus greater than 1 (equivalently, the hyperbolic
component in its Jordan decomposition is non-trivial).
\end{proof}

\section{${\rm{Aut}}(F_n)$ and ${\rm{Out}}(F_n)$ are not K\"ahler groups}\label{s:kahler}

One can deduce from Corollary \ref{c:evalues} that the standard representation 
${\rm{Aut}}(F_n)\to{\rm{GL}}(n,\Z)$ (given by the action of
${\rm{Aut}}(F_n)$ on $H_1(F_n,\Z)$) cannot be deformed locally to anything but a 
conjugate representation. A stonger result was proved by Dyer and Formanek \cite{DF}
(see also \cite{rapin} Theorem 1.2): 
every representation ${\rm{Aut}}(F_n)\to {\rm{GL}}(n,\mathbb C)$ factors through the standard representation.
Using Proposition 3.1 of \cite{BV11} one can extend the proof  in \cite{rapin} to cover
the unique index-2 subgroup ${\rm{SAut}}(F_n)
\subset{\rm{Aut}}(F_n)$.

Every finitely presented
group is the fundamental group of a closed symplectic manifold and of a closed complex manifold,
but not every finitely presented group is a K\"ahler group, i.e. the fundamental group of a closed K\"ahler manifold (see
\cite{A+} for context and references). For example, a K\"ahler group cannot split over a finite group
as an alamagamated free product or HNN extension, it cannot have a subgroup of finite index that
has odd first betti number, and it cannot be an extension of a 
group with infinitely many ends by a finitely generated group. This last
condition covers ${\rm{Aut}}(F_2)$ and ${\rm{Out}}(F_2)$.
But if $n\ge 3$
then ${\rm{Aut}}(F_n)$ has property FA and it seems likely (but is unproved) that all of its subgroups of
finite index have finite abelianization. There is, however,
 a more subtle obstruction coming from Simpson's work on
non-abelian Hodge theory that one can use to show that ${\rm{Aut}}(F_n)$
and ${\rm{Out}}(F_n)$ are not K\"ahler.

Simpson \cite{carlos} proves that if a group 
$\G$ admits a representation $\rho:\G\to {\rm{GL}}(n,\mathbb C)$
with image ${\rm{SL}}(n,\Z)$, where $n\ge 3$, and if this representation cannot be deformed locally into a non-conjugate 
representation, then $\Gamma$ is not K\"ahler. In more detail, if $\G$ were K\"ahler then the
real Zariski closure of $\rho(\G)$ would be a group of Hodge type (\cite{carlos},
Lemma 4.4), but ${\rm{SL}}(n,\R)$ is not of Hodge type ($n\ge 3$) --- see
\cite{carlos}, p.50--51.
We noted above that the standard representation of  ${\rm{SAut}}(F_n)$ is rigid.
 And a finite index subgroup of a K\"ahler group is (obviously) a K\"ahler group.

\bigskip
 \end{document}